\documentclass[11pt,reqno]{amsart}%{report}

\usepackage{mathrsfs} \usepackage{amssymb} \usepackage{amsmath} 
\usepackage{amsthm}
\usepackage[dvips]{graphicx, color}
\usepackage{enumerate}

\usepackage{epsfig}
\usepackage{subfigure}

\newcounter{ctr}
\newtheorem{definition}[ctr]{Definition}
\newtheorem{theorem}[ctr]{Theorem}
\newtheorem{proposition}[ctr]{Proposition}
\newtheorem{lemma}[ctr]{Lemma}

\newtheorem{corollary}[ctr]{Corollary}

\newcommand{\T}{\ensuremath{\mathbb{T}}}
\newcommand{\K}{\ensuremath{\mathscr{K}}}
\newcommand{\U}{\ensuremath{\mathscr{U}}}
\newcommand*{\R}{\ensuremath{\mathbb{R}}}
\renewcommand*{\S}{\ensuremath{\mathcal{S}}}
\newcommand*{\N}{\ensuremath{\mathbb{N}}}

\newcommand*{\C}{\ensuremath{\mathbb{C}}}
\newcommand*{\supp}{\ensuremath{\mathrm{supp\,}}}
\newcommand*{\dist}{\ensuremath{\mathrm{dist\,}}}

\renewcommand*{\Re}{\ensuremath{\mathrm{Re\,}}}

\newcommand{\eps}{\varepsilon}

\begin{document}

\title[Weak stationary solutions]{Weak solutions to the stationary incompressible Euler equations}

\author{A.~Choffrut}
\address{Maxwell Institute {\bf \&} School of Mathematics, University of Edinburgh, Edinburgh EH9 3JZ, UK}
\email{antoine.choffrut@ed.ac.uk}

\author{L.~Sz\'ekelyhidi Jr.}
\address{Institut f\"ur Mathematik, Universit\"at Leipzig, D-04103 Leipzig}
\email{laszlo.szekelyhidi@math.uni-leipzig.de}

\date{\today}            

\begin{abstract}
We consider weak stationary solutions to the incompressible Euler equations and show that the analogue of the
$h$-principle obtained by the second author in joint work with C.~De Lellis for time-dependent weak solutions in $L^\infty$ continues to hold.
The key difference arises in dimension $d=2$, where it turns out that the relaxation is strictly smaller 
than what one obtains in the time-dependent case.
\end{abstract}

\maketitle

%\tableofcontents

\bigskip

\hrule

\bigskip

\section{Introduction}
It is well-known since the work of V.~I.~Arnold that the Euler equations in 2 dimensions for ideal fluids exhibit a very rich geometric structure. This arises from the interpretation of the Euler equations as the equations of geodesics on the space of volume-preserving diffeomorphisms. In particular, coupled with the fact that in 2d the vorticity is transported by the flow, one obtains, at least formally, a very explicit geometric picture as a space of diffeomorphisms foliated by distributions of vorticities, and on each single leaf the equation can be thought of as a Hamiltonian system. 

A first step towards an analytic verification of this formal picture was taken in \cite{CHOFFRUT-SVERAK} for stationary solutions, i.e.~solutions of the system
\begin{equation}
(v\cdot \nabla) v+\nabla p=0\,,\qquad {\rm div}\,v=0\,.
\label{eq:stationary.Euler}
\end{equation}
Under some non-degeneracy assumptions it was shown that locally near each stationary solution there exists a manifold of stationary solutions transversal to the foliation. In analytical terms this amounts to an implicit function theorem, showing that there is locally a one-to-one correspondence between leaves of the foliation and solutions of \eqref{eq:stationary.Euler}. This is the geometric picture in the class of \emph{smooth} solutions of \eqref{eq:stationary.Euler}. 

In this short note we would like to explore an entirely different scenario, namely the picture suggested by Gromov's h-principle as applied to fluid mechanics in \cite{Bulletin}, implying that there is an abundant set of \emph{weak} stationary solutions in the neighbourhood of any smooth stationary solution. The fact that weak forms of the h-principle apply to the non-stationary Euler equations has been discovered in \cite{DL-Sz:Euler-as-differential-inclusion}, see also the survey \cite{Bulletin}.
Our main result is the following
\begin{theorem}\label{theorem:main.theorem}
Let $d\geq 2$ and $v_0$ a smooth stationary Euler flow on $\mathbb T^d$,
and consider a smooth function $e(x)>|v_0(x)|^2$ for $x\in\mathbb T^d$.
Then, for every $\sigma>0$, there exist infinitely many weak stationary flows
$v\in L^\infty(\mathbb T^d;\mathbb R^d)$
such that $|v(x)|^2=e(x)$ for a.e. $x\in\mathbb T^d$ and $\|v-v_0\|_{H^{-1}(\mathbb T^d)}<\sigma$.
\end{theorem}

\smallskip

Theorem~\ref{theorem:main.theorem}
should be seen as the natural counterpart to the h-principle obtained
in \cite{DL-Sz:Euler-as-differential-inclusion}
for $L^\infty$-solutions to  the {\it non-stationary} ({\it i.e.} time-dependent) Euler equations
$$\partial_tv+(v\cdot\nabla) v+\nabla p=0\,,\qquad {\rm div}\,v=0\,.$$
It turns out, however, that the methods that have been introduced for the non-stationary case do not directly transfer to the stationary case in 2d. In technical terms, the relaxation set obtained when passing from solutions to subsolutions is strictly smaller than the convex hull. 
%More precisely, assumption (H2) in \cite{SZEKELYHIDI:SISSA-lectures} is not satisfied by the set $\mathscr U_r$ defined in (\ref{e:Ur}).
%Thus, the stationary Euler equations in two dimensions are not covered by earlier work.
See Section \ref{s:wavecone} for a precise formulation. This observation resembles the rigidity results obtained by A.~Shnirelman \cite{Shnirelman} concerning the geometry of measure-preserving homeomorphisms in the 2d versus the much more flexible 3d case. 

On the other hand, for dimensions $d\geq 3$ one can essentially retain the framework developed in \cite{DL-Sz:Euler-as-differential-inclusion,DL-Sz:admissibility}.

\bigskip

We remark, that the approximation in Theorem \ref{theorem:main.theorem} can be taken in any negative Sobolev norm.
Recall also that $v\in L^\infty(\mathbb T^d;\mathbb R^d)$ is a \emph{weak solution}
to (\ref{eq:stationary.Euler})
if
$$\int_{\mathbb T^d} v\otimes v : \nabla \Phi\,dx=0\,\qquad \int_{\mathbb T^d} v\cdot \nabla f\,dx=0$$
for every divergence-free vector field $\Phi\in C^\infty(\mathbb T^d;\mathbb R^d)$ 
and every scalar function $f\in C^\infty(\mathbb T^d)$. Finally, concerning the pressure we note that, using the equation
$\Delta p=-\textrm{div}[(v\cdot\nabla)v]=-\textrm{div }\textrm{div }(v\otimes v)$, the pressure $p$ can be recovered using standard estimates as a function $p\in L^q(\T^d)$ for all $q<\infty$. In fact, as in \cite{DL-Sz:Euler-as-differential-inclusion}
one can even construct $p\in L^\infty(\T^d)$, but we will not pursue this further in this paper.
 
\bigskip

We note in passing that in the time-dependent case \cite{DL-Sz:continuous-Euler-flows, ISETT, B-DL-Sz} has lead to solutions with H\"older regularity, a question that has been the focus of interest in view of Onsager's conjecture on anomalous dissipation in turbulence. However, the methods of \cite{DL-Sz:continuous-Euler-flows, DL-Sz:Holder, ISETT, B-DL-Sz} do not apply to the stationary case.
Indeed, a very delicate part in these proofs is to use the transport operator
$\partial_t+v\cdot\nabla$ to absorb the main (linear) part of the error in the iteration.
Although the stationary case is not directly related to Onsager's conjecture, 
there is a natural analogue of the problem for H\"older-continuous stationary flows \cite{Cheskidov-Shvydkoy}.

\section{The reformulation as a differential inclusion}
\label{section:general.considerations}

Our proof of Theorem~\ref{theorem:main.theorem} is based on the convex integration framework for the Euler equations, as
developed in \cite{DL-Sz:Euler-as-differential-inclusion}. For the convenience of the reader we recall the setting in this section, specializing on the time-independent case.

We denote by 
$$
\S^d_0=\left\{u\in\mathbb R^{d\times d}~\colon~u^\top = u,\quad {\rm tr}\,u=0\right\}
$$
the set of symmetric, trace-free $d\times d$-matrices. By $|u|$ we shall mean the operator norm of $u\in\S^d_0$.
The following is elementary.

\begin{lemma}\label{lemma:linearization}
Let $d\geq 2$.
Let $e\in C(\mathbb{T}^d)$ be a positive function.
Suppose $v\in L^\infty(\T^d;\R^d)$, $u\in L^\infty(\T^d;\S^{d}_0)$,
and $q\in \mathcal{D}'(\T^d)$ a distribution solve weakly
\begin{equation}
{\rm div}\,u+\nabla\,q=0,\qquad{\rm div}\,v=0\,.
\label{e:linear}
\end{equation}
If 
\begin{equation}\label{e:pointwise}
u=v\otimes v-\frac{e}d\,{\rm Id}\quad\textrm{  a.e. in $\T^d$},
\end{equation}
then $v$ and $p:=q-\frac{e}d$ solve (\ref{eq:stationary.Euler})
weakly, and $|v(x)|^2=e(x)$ for a.e. $x\in\mathbb{T}^d$.
\end{lemma}

We will call a pair $w=(v,u):\T^d\to\R^d\times \S^{d}_0$ a \emph{stationary subsolution}, if there exists a distribution $q\in \mathcal{D}'(\T^d)$ such that the triple $(v,u,q)$ is a weak solution of \eqref{e:linear} (cf. \cite{DL-Sz:admissibility} Section 3.1 and \cite{Bulletin} Section 4). 

Lemma \ref{lemma:linearization} allows us to formulate the problem as a differential inclusion. For any $r>0$ let
\begin{equation}
\mathscr K_r:=\left\{(v,u)\in \R^d\times \S^{d}_0~\colon~
u=v\otimes v-\frac{r}d\,{\rm Id}\right\}\subset \R^{d_*}
\qquad (d\geq 2)
\label{eq:the.constraint.set.K.first.formulation}
\end{equation}
where 
$$
d_*=\frac{d(d+1)}2-1.
$$
Note that for each $r>0$ the set $\K_r$ is a compact, smooth submanifold of $\R^{d_*}$ of dimension $d$. 
A  weak solution to the Euler equations (\ref{eq:stationary.Euler})
with energy profile $e(x)$ is therefore 
(identified with) a subsolution $w=(v,u)$ 
which satisfies the pointwise inclusion
\begin{equation}\label{e:inclusion}
w(x)\in\mathscr K_{e(x)}\qquad {\rm for~ {\it a.e.}}~ x\,.
\end{equation}
The idea is to relax the constraint set $\K_{e(x)}$ in \eqref{e:inclusion} to a suitable nonempty open subset of the convex hull: 
$$
\U_{e(x)}\subset \K_{e(x)}^{co}.
$$  
The key property required of the sets $\U_{r}\subset \K_r^{co}$ is the following, based on the notion of {\it stability of gradients} introduced by B.~Kirchheim in Section 3.3 of \cite{KIRCHHEIM:habilitation} 
(see also \cite{Sz:IPM,SZEKELYHIDI:SISSA-lectures}).

\bigskip

\noindent{\bf Perturbation Property (P): }There is a continuous strictly increasing function $\Phi:[0,\infty)\to[0,\infty)$ with $\Phi(0)=0$ with the following property.
Let $Q=(0,1)^d$ be the open unit cube in $\R^d$.
For every $\bar{w}:=(\bar{v},\bar{u})\in\U_r$ there exists a subsolution $w=(v,u)\in C_c^\infty(Q;\R^d\times\S^d_0)$ with associated pressure $q\in C_c^{\infty}(Q)$ such that
\begin{itemize}
\item $\bar{w}+w(x)\in\U_r$ for all $x\in Q$;
\item $\int_{Q}|w(x)|^2\,dx\geq \Phi(\dist(\bar{w},\K_r))$.
\end{itemize}

\bigskip

In \cite{DL-Sz:admissibility} it was shown that, in the case of the time-dependent Euler equations, 
the perturbation property is satisfied with $\U_r=\textrm{int }\K_r^{co}$, and the convex hull was explicitly calculated
\begin{equation}\label{e:convexhull}
\K^{co}_r=\left\{(v,u)\in \mathbb R^d\times \S^{d}_0~\colon~
v\otimes v-u\leq \frac{r}d\,{\rm Id}\right\}.
\end{equation}
A useful consequence of this formula is that, provided $\bar{w}=(\bar{v},\bar{u})\in\K_r^{co}$, we have
$|\bar{v}|^2=r$ implies $\bar{w}\in\K_r$. Consequently there exists a continuous strictly increasing function $\Psi:[0,\infty)\to[0,\infty)$ with $\Psi(0)=0$ such that
\begin{equation}\label{e:Psi}
\dist(\bar{w},\K_r)\leq \Psi(r-|\bar{v}|^2)\quad\textrm{ for all }\bar{w}=(\bar{v},\bar{u})\in\K_r^{co}.
\end{equation}
Hence, in property (P) we may replace $\dist(\bar{w},\K_r)$ by $r-|\bar{v}|^2$. 

It turns out the the arguments used in \cite{DL-Sz:admissibility} are insufficient to deal analogously with the stationary case - the main reason is that, while the constraint set $\K_r$ is the same in both cases, the associated wave-cone $\Lambda$ (see Section \ref{s:wavecone} below) is smaller in the  stationary case. In fact as a result it turns out that in the 2-dimensional situation (P) is not satisfied with $\U_r=\textrm{int }\K_r^{co}$ (see Section \ref{s:failureofP}).

\bigskip

If property (P) is satisfied for some family of open sets $\U_r$, $r>0$, the by now standard Baire-category argument leads to the existence of a residual set of weak solutions. In order to obtain the precise statement of Theorem \ref{theorem:main.theorem} we require, in addition to (P), the following:
\begin{equation}\label{e:monotone}
\K_r\subset \U_{r'}\quad\textrm{ for }r<r'.
\tag{*}
\end{equation}
Property (*) will ensure that smooth stationary flows belong to the set of subsolutions given by the relaxed set $\U_r$, see Step 2 of the proof of Theorem \ref{theorem:main.theorem}.

We now sketch the argument for the convenience of the reader, but wish to emphasize that this proof is by now standard.

\noindent{\bf Proof of Theorem \ref{theorem:main.theorem}, assuming (P) and \eqref{e:monotone}. }

{\bf Step 1: The functional analytic setup. }
Let $e=e(x)>0$ be a positive smooth function, and define
$$
X_0=\left\{w\in C^\infty(\T^d;\R^d\times\S^d_0):\,w\textrm{ subsolution such that }w(x)\in \mathscr{U}_{e(x)}\textrm{ for all }x\in\T^d\right\}.
$$
It is not difficult to check that $X_0$ is bounded in $L^2(\T^d)$. Indeed, let $\bar{e}=\max_{x\in\T^d}e(x)$ and observe that, if $w(x)=(v(x),u(x))\in \U_{e(x)}\subset \K_{\bar{e}}^{co}$ (using \eqref{e:convexhull}), then
$$
|v|^2\leq \bar{e},\quad |u|\leq 2\bar{e},
$$
and hence $w=(v,u)\in X_0$ implies $\|v\|_{L^\infty}^2,\,\|u\|_{L^\infty}\leq 2\bar{e}$. Standard elliptic estimates and the equation $\textrm{div }\textrm{div }u=-\Delta q$ then imply that $\|q\|_{L^2}\leq C\bar{e}$. See also Lemma 6.5 in \cite{SZEKELYHIDI:SISSA-lectures}. 
We define $X$ to be the closure of $X_0$ in the weak $L^2$ topology (which is metrizable by the boundedness). 

\medskip

{\bf Step 2: $X$ contains smooth stationary flows. }
Let $v_0$ be a smooth solution of \eqref{eq:stationary.Euler} with (smooth) pressure $p_0$ and let $e=e(x)$ be a smooth function
such that $e(x)>|v_0(x)|$ for all $x\in\T^d$. Let 
$$
u_0=v_0\otimes v_0-\frac{|v_0|^2}{d}{\rm Id},\quad q_0=p_0+\frac{|v_0|^2}{d}.
$$
By definition $w_0=(v_0,u_0)$ is a subsolution and 
$$
w_0(x)\in \K_{|v_0|^2(x)}\quad\textrm{ for all }x\in \T^d.
$$
Assumption \eqref{e:monotone} then implies that $w_0(x)\in \U_{e(x)}$ for all $x\in\T^d$, hence
$$
w_0\in X_0.
$$

\medskip

{\bf Step 3: Continuity points of $w\mapsto \int|w|^2\,dx$. }We note that the mapping $w\mapsto \int|w|^2\,dx$ is a Baire-1 map in $X$, hence its continuity points form a residual set in $X$. On the other hand property (P) with an easy covering and rescaling argument leads to the following: there exists a continuous strictly increasing function $\tilde\Phi:[0,\infty)\to[0,\infty)$ with $\tilde\Phi(0)=0$ such that, for every $w\in X_0$ there exists a sequence $w_k\in X_0$ such that 
\begin{itemize}
\item $w_k\rightharpoonup w$ weakly in $L^2(\T^d)$;
\item $\int_{\T^d}|w_k-w|^2\,dx\geq \tilde\Phi\left(\int_{\T^d}\dist(w(x),\K_{e(x)})\right)$.
\end{itemize}
(For instance, one may take $\tilde\Phi$ to be the convex envelope of $\Phi$ - up to rescaling.
See \cite{DL-Sz:admissibility, SZEKELYHIDI:SISSA-lectures,Sz:IPM}).
Consequently, using a diagonal argument and the metrizability of $X$, see \cite{KIRCHHEIM:habilitation,DL-Sz:admissibility, SZEKELYHIDI:SISSA-lectures}, continuity points of the map $w\mapsto \int|w|^2\,dx$ in $X$ are subsolutions $w$ such that $w(x)\in \K_{e(x)}$ for almost every $x\in\T^d$. Since a residual set in $X$ is dense, there exist a sequence $w_k=(v_k,u_k)\in X$ with $w_k(x)\in \K_{e(x)}$ a.e., such that $w_k\rightharpoonup w_0$. In particular this means that $v_k$ is a weak stationary solution of the Euler equations with $|v_k(x)|^2=e(x)$ for a.e. $x\in\T^d$. 

\qed

\bigskip

The rest of the paper is thus devoted to constructing a family of open sets $\U_r$ with the properties (P) and \eqref{e:monotone}.
The perturbation property (P) requires a large class of subsolutions with specific oscillatory behaviour at our disposal. In Section \ref{s:wavecone} we show how such stationary subsolutions can be constructed in general dimension $d\geq 2$, based on the notion of laminates of finite order. Then, in Sections \ref{s:d3} and \ref{s:d2} we will treat separately the cases $d\geq 3$ and $d=2$, respectively.

\section{The wave-cone and laminates}
\label{s:wavecone}

To the linear system \eqref{e:linear}
we associate the \emph{wave cone} $\Lambda$ defined as the set 
\begin{equation}
\Lambda = \biggl\{(\bar{v},\bar{u})\in(\R^d\setminus\{0\})\times \S^{d}_0~\colon~
\exists\,\bar{q}\in\R,\,\eta\in\R^d\setminus\{0\}\textrm{ s.t. }\bar{u}\eta+\bar{q}\eta=0\,, \bar{v}\cdot\eta=0\biggr\}\,.
\label{eq:Lambda.in.any.dimension}
\end{equation}
As in \cite{DL-Sz:Euler-as-differential-inclusion}, this set corresponds to plane-wave solutions of \eqref{e:linear}. Note that, in contrast with the time-dependent case, here we have $\Lambda\neq (\R^d\setminus\{0\})\times \S^{d}_0$ (for the equality in the time-dependent case, see Remark 1 in \cite{DL-Sz:Euler-as-differential-inclusion}).  Nevertheless, we can localize plane-waves by using the same potentials as in the time-dependent case, by simply restricting to potentials which are independent of time. We obtain:

\begin{lemma}\label{l:potential}
Let $d\geq 2$.
Let $\bar{w}=(\bar{v},\bar{u})\in\Lambda$. Then 
\begin{enumerate}
\item $\exists\,\eta\in\R^d\setminus\{0\}$ such that for any $h\in C^{\infty}(\R)$ 
$$
w(x):=\bar{w}h(x\cdot \eta)
$$
is a subsolution;
\item There exists a second order homogeneous linear differential operator $\mathcal{L}_{\bar{w}}$ such that
$$
w:=\mathcal{L}_{\bar{w}}[\phi]
$$
is a subsolution for any $\phi\in C^{\infty}(\R^d)$;
\item Moreover, if $\phi(x)=H(x\cdot\eta)$ for some $H\in C^{\infty}(\R)$, then
$$
\mathcal{L}_{\bar{w}}[\phi](x)=\bar{w}H''(x\cdot\eta).
$$
\end{enumerate}
\end{lemma}

\begin{proof}
See \cite{DL-Sz:Euler-as-differential-inclusion} Proposition 3.2 and \cite{SzW} Proposition 20. 
\end{proof}

Lemma \ref{l:potential} allows us to construct stationary subsolutions with specific oscillatory behaviour. For the time-dependent Euler equations this was done in Section 3.3 of \cite{SzW}. In the following we denote by $Q=(0,1)^d$ the open unit cube in $\R^d$, and for $w_1,w_2\in \R^d\times\S^{d}_0$ by $[w_1,w_2]:=\{\lambda w_1+(1-\lambda)w_2:\,\lambda\in[0,1]\}$ the line segment joining $w_1$ and $w_2$. 

\begin{lemma}\label{l:2states}
Let $d\geq 2$.
Let $w_i=(v_i,u_i)\in \R^d\times\S^{d}_0$ and $\mu_i\geq 0$ such that
$$
w_2-w_1\in\Lambda,\quad \mu_1w_1+\mu_2w_2=0,\quad \mu_1+\mu_2=1.
$$
For any $\eps>0$ there exists a subsolution $w\in C_c^{\infty}(Q;\R^d\times\S^{d}_0)$ such that
\begin{enumerate}
\item[(i)] $\dist(w(x),[w_1,w_2])<\eps$ for all $x\in Q$;
\item[(ii)] There exist disjoint open subsets $A_1,A_2\subset Q$ such that for $i=1,2$ 
$$
w(x)=w_i\textrm{ for all $x\in A_i$,}\qquad \left||A_i|-\mu_i\right|<\eps.
$$
\end{enumerate}
\end{lemma}

\bigskip

Using Lemma \ref{l:2states} as the basic building-block, more complicated oscillatory behaviour can be achieved. The key concept is the notion of \emph{laminates of finite order} \cite{MUELLER-SVERAK} (called \emph{prelaminates }Êin \cite{KIRCHHEIM:habilitation}). We recall 

\begin{definition}\label{d:laminate}
Let $d\geq 2$.
Let $\U\subset \R^d\times\S^d_0$ be a set. The set of \emph{laminates of finite order}, denoted by $\mathcal{L}(\U)$, is the smallest class of (atomic) probability measures supported on $\U$ that
\begin{itemize}
\item contains all Dirac-masses supported on $\U$;
\item is closed under splitting along $\Lambda$-segments inside $\U$.
\end{itemize}
The latter means the following: if $\nu=\sum_{i=1}^N\nu_i\delta_{w_i}\in \mathcal{L}(\U)$, and $w_N\in [z_1,z_2]\subset \U$ with $z_2-z_1\in \Lambda$, then 
$$
\sum_{i=1}^{N-1}\nu_i\delta_{w_i}+\nu_N\left(\lambda\delta_{z_1}+(1-\lambda)\delta_{z_2}\right)\,\in\mathcal{L}(\U),
$$
where $\lambda\in[0,1]$ such that $w_N=\lambda z_1+(1-\lambda) z_2$. 
\end{definition}

A simple induction argument and Lemma \ref{l:2states} then leads to
\begin{proposition}\label{p:laminates}
Let $d\geq 2$.
Let $\U\subset\R^d\times\S^d_0$ be open and
$$
\nu=\sum_{i=1}^N\mu_i\delta_{w_i}\in \mathcal{L}(\U)
$$
be a laminate of finite order with barycenter $\overline{\nu}=0$. For any $\eps>0$ there exists a subsolution $w\in C_c^{\infty}(Q;\R^d\times\S^{d}_0)$ such that
\begin{enumerate}
\item[(i)] $w(x)\in \U$ for all $x\in Q$;
\item[(ii)] there exist pairwise disjoint open subsets $A_1,\dots,A_N\subset Q$ such that for $i=1,\dots,N$
$$
w(x)=w_i\textrm{ for all $x\in A_i$,}\qquad \left||A_i|-\mu_i\right|<\eps.
$$
\end{enumerate}
\end{proposition}

\bigskip

In light of Proposition \ref{p:laminates} we obtain immediately a useful \emph{sufficient condition} for Property (P):
\begin{proposition}\label{p:criterion}
Let $d\geq 2$.
Let $\U_r\subset \K_r^{co}$ an open set with the following property: there exists a continuous strictly increasing function $\Phi:[0,\infty)\to[0,\infty)$ with $\Phi(0)=0$ such that for any $w_0\in\U_r$ there exists a laminate of finite order $\nu\in \mathcal{L}(\U_r)$ with barycenter $\bar{\nu}=w_0$ such that
$$
\int|w-w_0|^2\,d\nu(w)\geq \Phi(\dist(w_0,\K_r)).
$$
Then $\U_r$ has Property (P).
\end{proposition}
 
\bigskip

Next, we recall the definition of the lamination-convex hull of a set. 

\bigskip

\begin{definition}\label{d:lchull}
Let $d\geq 2$.
Let $\U\subset\R^d\times \S^d_0$ be a set. The \emph{lamination convex hull} $\U^{lc}$ (with respect to the wave-cone $\Lambda$) is defined as
$$
\U^{lc}=\bigcup_{i=0}^\infty \U^{(i)},
$$
where $\U^{(i)}$ is defined inductively as: $\U^{(0)}=\U$ and
$$
\U^{(i+1)}:=\U^{(i)}\cup\left\{ t\xi+(1-t)\xi'~\colon~\xi,\xi'\in \U^{(i)}, 
\xi-\xi'\in\Lambda, t\in[0,1]\right\}\,.
$$
\end{definition}
Note that in general we have $\U^{lc}\subseteq \U^{co}$. The following is an elementary consequence of Definitions \ref{d:laminate} and \ref{d:lchull}: 

\begin{lemma}\label{l:openlchull}
Let $d\geq 2$.
Let $\U\subset \R^{d}\times\S_0^d$ be an open set. Then $\U^{lc}$ is open. Moreover, for any $w\in \U^{lc}$ there exists 
a laminate of finite order $\nu\in \mathcal{L}(\U^{lc})$ with barycenter $\bar{\nu}=w$ such that $\supp \nu\subset\U$.
\end{lemma}

We will see in the sections below that, for the set $\K_r$ corresponding to the Euler equations \eqref{eq:the.constraint.set.K.first.formulation} with the wave cone $\Lambda$ in \eqref{eq:Lambda.in.any.dimension} (corresponding to \emph{stationary} solutions), we have  
\begin{itemize}
\item If $d\geq 3$, $\K_r^{lc}=\K_r^{co}$ (c.f. Section \ref{s:d3});
\item If $d=2$, $\K_r^{lc}\subsetneq\K_r^{co}$ (c.f. Section \ref{s:failureofP}).
\end{itemize}
In particular, in the case $d\geq 3$ one can essentially reduce to the case of time-dependent solutions as done in \cite{DL-Sz:Euler-as-differential-inclusion,DL-Sz:admissibility}. On the other hand for $d=2$ we will need to construct an explicit set $\U_r$ satisfying the perturbation property (P) in Section \ref{s:d2}. This will require a more careful analysis of compatible oscillations, more precisely an analysis of laminates of finite order.

%%%%%%%%%%%%%%%%%%%%%%%%%%%%%%%%%%%%%%%%%%%%%%%
%%%%%%%%%%%%%%%%%%%%%%%%%%%%%%%%%%%%%%%%%%%%%%%
\section{The case $d\geq 3$}
\label{s:d3}

Let us first consider the case of dimension $d\geq 3$. It turns out that in this case the proof of Theorem \ref{theorem:main.theorem} can be essentially reduced to the time-dependent case.

We recall some terminology from \cite{DL-Sz:admissibility} Section 4.3. Given $r>0$ we call a line segment $\sigma\subset\R^d\times\S_0^d$ \emph{admissible} if 
\begin{itemize}
\item $\sigma$ is contained in the interior of $\K_r^{co}$;
\item $\sigma$ is parallel to $(a,a\otimes a)-(b,b\otimes b)$ for some $a,b\in\R^d$ with $|a|^2=|b|^2=r$ and $b\neq \pm a$.
\end{itemize}
We have the following:
\begin{lemma}[Lemma 6 in \cite{DL-Sz:admissibility}]\label{l:admissiblesegments}
Let $d\geq 2$. There exists a constant $C=C(d,r)>0$, such that for any $w=(v,u)\in\textrm{int }\K^{co}_r$ there exists an admissible line segment $\sigma=[w-\bar{w},w+\bar{w}]$, $\bar{w}=(\bar{v},\bar{u})$, such that
$$
|\bar{v}|\geq C\left(r-|v|^2\right)\textrm{ and }\dist(\sigma,\partial\K^{co}_r)\geq \frac{1}{2}\dist(w,\partial\K^{co}_r).
$$
\end{lemma}

\medskip

The key observation is that, even in the stationary case with $d\geq 3$, admissible line segments are in $\Lambda$-directions:

\begin{lemma}\label{l:Lambda3}
Let $d\geq 3$.
Let $a,b\in\R^d$ with $|a|^2=|b|^2=r$ and $b\neq \pm a$, and let $(\bar{v},\bar{u})=(a,a\otimes a)-(b,b\otimes b)$. Then $(\bar{v},\bar{u})\in\Lambda$.
\end{lemma}

\begin{proof}
Recall from \eqref{eq:Lambda.in.any.dimension} that $(\bar{v},\bar{u})\in \Lambda$ if there exists a vector $\eta\neq 0$ such that $\bar{v}\cdot \eta=0$ and $\bar{u}\eta=\bar{q}\eta$ for some $\bar{q}\in\R$. Choose $\eta\in\R^d\setminus\{0\}$ such that $\eta\cdot a=\eta\cdot b=0$. Then obviously $\eta\cdot\bar{v}=0$ and $\bar{u}\eta=(a\otimes a-b\otimes b)\eta=0$. This proves that $(\bar{v},\bar{u})\in \Lambda$ (with $\bar{q}=0$).
\end{proof}

\medskip

\begin{corollary}\label{c:d3}
Let $d\geq 3$. Then $\U_r:=\textrm{int}\,\K^{co}_r$ has the perturbation property (P) and property \eqref{e:monotone}.
\end{corollary}

\begin{proof}
Property \eqref{e:monotone} follows easily from the explicit formula \eqref{e:convexhull}. 

To show property (P), let $\bar{w}\in\U_r$. Using Lemmas \ref{l:Lambda3} and \ref{l:admissiblesegments} we find the existence of 
$\tilde w\in\Lambda$, such that
$$
[\bar{w}-\tilde w,\bar{w}+\tilde w]\subset\U_r,\qquad \dist([\bar{w}-\tilde w,\bar{w}+\tilde w],\partial\K^{co}_r)\geq \frac{1}{2}\dist(\bar{w},\partial\K^{co}_r),
$$
and
$$
|\tilde w|\geq \frac{1}{4}C(r-|\bar v|^2).
$$
Using Lemma \ref{l:2states} with a suitable $\eps<\frac{1}{4}\dist(\bar{w},\partial\K^{co}_r)$ we construct a subsolution $w=(v,u)\in C_c^{\infty}(Q;\R^d\times\S^d_0)$ such that $\bar{w}+w(x)\in\U_r$ for all $x\in Q$ and
$$
\int_Q|w(x)|^2\,dx\geq \frac{1}{2}|\tilde w|^2\geq C'(r-|\bar v|^2)^2
$$
for some constant $C'>0$. Using the observation in \eqref{e:Psi} we deduce property (P) 
as required.
\end{proof}

%%%%%%%%%%%%%%%%%%%%%%%%%%%%%%%%%%%
%%%%%%%%%%%%%%%%%%%%%%%%%%%%%%%%%%%

\section{Laminates in the two-dimensional case}
\label{s:d2}

Let us now consider the case $d=2$. Given a vector $\bar{v}\in \R^2\setminus\{0\}$ we denote by $\bar{v}^\perp=(\bar{v}_2,-\bar{v}_1)$ the perpendicular.

We start with the following observation:
\begin{lemma}
Let $d=2$. Then $\Lambda$ in \eqref{eq:Lambda.in.any.dimension} can be written as
$$
\Lambda=\Bigl\{(\bar{v},\bar{u})\in(\R^2\setminus\{0\})\times\S^2_0:\,\bar{u}\bar{v}\cdot\bar{v}^{\perp}=0\Bigr\}.
$$
\end{lemma}

\begin{proof}
According to \eqref{eq:Lambda.in.any.dimension}, $(\bar{v},\bar{u})\in\Lambda$ precisely if $\bar{u}$ possesses an eigenvector perpendicular to $\bar{v}$. In two dimensions this means that $\bar{v}^\perp$ is an eigenvector of $\bar{u}$. The claim follows.
\end{proof}

\subsection{Suitable coordinates in state-space}
 
We proceed by introducing coordinates on the state-space $\R^2\times\S^{2}_0$.
The state variables $(v,u)$ can be written in coordinates as
$$
v=\begin{pmatrix}a\\b\end{pmatrix},\quad u=\begin{pmatrix}c&d\\d&-c\end{pmatrix}.
$$
It is then convenient to identify the state space $\R^2\times\S^{2}_0$ with $\C\times\C$, by introducing
$$
z=a+ib,\quad \zeta=c+id,
$$
so that, in the following, we will write
$$
w=(z,\zeta)\in \R^2\times\S^{2}_0.
$$
In these variables we have
\begin{align*}
\K_r&=\Bigl\{(z,\zeta):\,|z|^2=r\textrm{ and }\zeta=\tfrac{1}{2}z^2\Bigr\}\\
\Lambda&=\Bigl\{(z,\zeta):\,\Im(z^2\bar{\zeta})=0\Bigr\}.
\end{align*}
It is easy to see that both $\K_r$ and $\Lambda$ are invariant under the transformations
\begin{equation}\label{e:rotation}
R_{\theta}:(z,\zeta)\mapsto(ze^{i\theta},\zeta e^{2i\theta}),\qquad \theta\in[0,2\pi],
\end{equation}
and 
\begin{equation}\label{e:bar}
(z,\zeta)\mapsto (\bar{z},\bar{\zeta}).
\end{equation}
In light of \eqref{e:rotation} it is natural to consider the 3-dimensional subspace
$$
L=\Bigl\{(z,\zeta)\in\C\times\C:\,\Im(\zeta)=0\Bigr\}
$$
where we can use the coordinates $(a+ib,c)\in \C\times \R\cong L$. Note that in these coordinates
$$
\K_r\cap L=\Bigl\{(\sqrt{r},\tfrac{1}{2}r),(-\sqrt{r},\tfrac{1}{2}r),(i\sqrt{r},-\tfrac{1}{2}r),(-i\sqrt{r},-\tfrac{1}{2}r)\Bigr\}
$$
and
\begin{equation}\label{e:LambdaL}
\Lambda\cap L=\Bigl\{(a+ib,c):\,abc=0\Bigr\}.
\end{equation}

\subsection{Laminates in $L$}

We begin with an explicit construction. Fix $r>0$. We define for $(a+ib,c)$ with $|c|<r/2$ 
$$
f_r(a+ib,c):=\frac{\sqrt{r}|a|}{\frac{r}{2}+c}+\frac{\sqrt{r}|b|}{\frac{r}{2}-c}
$$
and set
\begin{equation}\label{e:Vr}
V_r=\Bigl\{(z,c)\in L:\,f_r(z,c)<1,\, |c|<r/2\Bigr\}.
\end{equation}

\begin{figure}[h]
\includegraphics[scale=0.9]{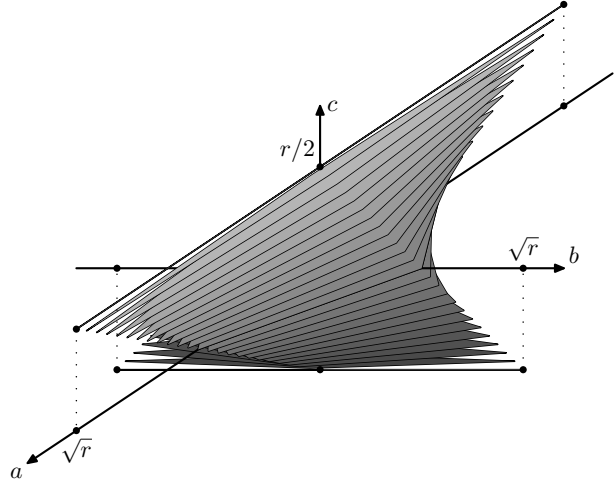}
\caption{\label{f:Vr} The set $V_r$, bounded by 4 ruled surfaces - see Proposition \ref{p:Vr} (iv).}
\end{figure}
The sets $V_r$, $r>0$, have following properties:
 
\begin{proposition}\label{p:Vr}
For any $r>0$ we have
\begin{enumerate}
\item[(i)] $V_r$ is open (relatively in $L\cong \C\times\R$);
\item[(ii)] $\overline{V_r}\supset (\K_r\cap L)$:
\item[(iii)] $\overline{V_{r'}}\subset V_r$ for any $0<r'<r$;
\item[(iv)] $\overline{V_r}\subset (\K_r\cap L)^{lc}$.
\end{enumerate}
More precisely, for any $w\in \overline{V_r}$ there exists a laminate of at most fourth order 
$\nu\in \mathcal{L}(\overline{V_r})$ such that
$\bar{\nu}=w$ and $\supp\nu\subset \K_r\cap L$.
\end{proposition}

\begin{proof}

{\bf (i) and (ii). }
The assertions (i) and (ii) are elementary after one observes that $f_r$ is a continuous function on $\{(z,c)\in L:\,|c|<r/2\}$ and
\begin{equation}\label{e:Vrbar}
\begin{split}
\overline{V_r}&=\Bigl\{f_r(a,b,c)\leq 1,\, |c|< r/2\Bigr\}\cup\\
&\cup \Bigl\{|a|\leq \sqrt{r},\, b=0,\,c=r/2\Bigr\}\cup \Bigl\{a=0,|b|\leq \sqrt{r},\, c=-r/2\Bigr\}.
\end{split}
\end{equation}

{\bf (iii). }Note that
$$
\frac{\partial}{\partial r}f_r(a+ib,c)=-\frac{(\frac12 r-c)}{2\sqrt{r}(\frac12 r+c)^2}|a|-\frac{(\frac12 r+c)}{2\sqrt{r}(\frac12 r-c)^2}|b|<0
$$
provided $|c|<\tfrac12r$ and $|a|+|b|>0$. Now let $(a+ib,c)\in \overline{V_{r'}}$ for some $r'<r$. Then $f_{r'}(a+ib,c)\leq 1$. 
If $|a|+|b|\neq 0$ and $|c|<\tfrac12r'$, we see that the function $r\mapsto f_r(a+ib,c)$ is strictly monotonic decreasing, consequently $f_r(a+ib,c)<1$ and hence $(a+ib,c)\in V_r$. If on the other hand $|a|+|b|=0$ and $|c|\leq \tfrac12r'$, then $f_{r}(a+ib,c)=0$ so that again $(a+ib,c)\in V_r$.

\smallskip

Finally, consider the case when $(a+ib,c)\in\overline{V_{r'}}$ and $|c|=\tfrac{1}{2}r'$. If $c=\tfrac{1}{2}r'$, using \eqref{e:Vrbar} we deduce $|a|\leq \sqrt{r'}$ and $b=0$, from which it is easy to deduce that $f_r(a+ib,c)<1$ by direct calculation. Similarly if $c=-\tfrac{1}{2}r'$. In both cases we see that $(a+ib,c)\in V_r$. This concludes the proof of (iii).

\bigskip

{\bf (iv). }Let $(a+ib,c)\in V_r$. Then $|c|<r/2$ and, on the (horizontal) $c$-slice the point $(a,b)$ lies inside the rhombus defined by the equation
$$
|a|\left(\tfrac{\sqrt{r}}{2}-\tfrac{c}{\sqrt{r}}\right)+|b|\left(\tfrac{\sqrt{r}}{2}+\tfrac{c}{\sqrt{r}}\right)\leq \left(\tfrac{r}{4}-\tfrac{c^2}{r}\right).
$$
\begin{figure}[h]
\includegraphics[scale=0.8]{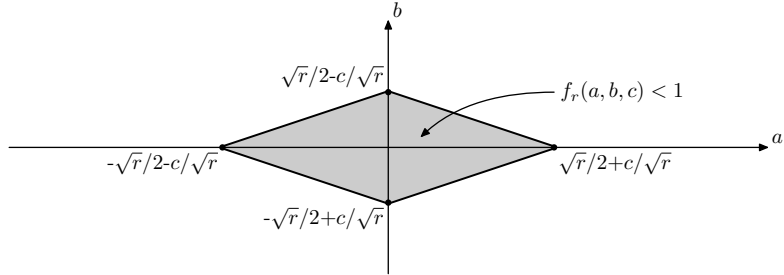}
\caption{\label{f:rhombus} The rhombus arising as a $c$-slice of $V_r$ with $0<c<r/2$.}
\end{figure}
Since any direction of the form $(\bar{a}+i\bar{b},0)$ is contained in $\Lambda$ (c.f. \eqref{e:LambdaL}), we find two points $(a_1+ib_1,c)$ and $(a_2+ib_2,c)$ on the boundary of the rhombus, so that the line segment joining the two points contains $(a+ib,c)$ and is in a $\Lambda$-direction. Therefore it suffices to show that the assertion holds for $(a+ib,c)\in \partial V_r$. 

\smallskip

Let $(a+ib,c)\in \partial V_r$. Using \eqref{e:rotation} and \eqref{e:bar} we may assume without loss of generality that $a,b\geq 0$, so that we have $|c|\leq r/2$ and 
$$
a\left(\tfrac{\sqrt{r}}{2}-\tfrac{c}{\sqrt{r}}\right)+b\left(\tfrac{\sqrt{r}}{2}+\tfrac{c}{\sqrt{r}}\right)=\left(\tfrac{r}{4}-\tfrac{c^2}{r}\right).
$$
It is easy to see that then $(a,b,c)$ lies on the (horizontal, hence $\Lambda$-) line segment connecting the two points 
$$
\Bigl(\frac{\sqrt{r}}{2}+\frac{c}{\sqrt{r}},0,c\Bigr)\quad\textrm{ and }\quad \Bigl(0,\frac{\sqrt{r}}{2}-\frac{c}{\sqrt{r}},c\Bigr)\,.
$$
Also,
\begin{align}
\Bigl(\frac{\sqrt{r}}{2}+\frac{c}{\sqrt{r}},0,c\Bigr)&\in\Biggl[(\sqrt{r},0,\frac{r}{2}),(0,0,-\frac{r}{2})\Biggr]\label{e:line3}\\
\Bigl(0,\frac{\sqrt{r}}{2}-\frac{c}{\sqrt{r}},c\Bigr)&\in\Biggl[(0,0,\frac{r}{2}),(0,\sqrt r,-\frac{r}{2})\Biggr]\label{e:line4}
\end{align}
and 
\begin{align}
(0,0,\frac{r}{2})&\in \Bigl[(-\sqrt{r},0,\frac{r}{2}),(\sqrt{r},0,\frac{r}{2})\Bigr]\label{e:line5}\\
(0,0,-\frac{r}{2})&\in \Bigl[(0,-\sqrt{r},-\frac{r}{2}),(0,\sqrt{r},-\frac{r}{2})\Bigr]\label{e:line6}
\end{align}
Using \eqref{e:LambdaL} we check that the line segments in \eqref{e:line3}-\eqref{e:line6} are in $\Lambda$-directions. 
Consequently $(a+ib,c)\in (\K_r\cap L)^{lc}$. The statement of the Proposition follows easily. 
\end{proof}

\subsection{Construction of $\U_r$}

Let $r>0$ and set
\begin{equation}\label{e:Ur}
\mathcal{V}_r=\Bigl\{(ze^{i\theta},ce^{2i\theta})\in \C\times\C:\,(z,c)\in V_r,\,0<|c|<\frac{r}{2},\,\theta\in\R\Bigr\},\quad \U_r:=\mathcal{V}_r^{lc}.
\end{equation}
Observe that, although in the definition of $\mathcal{V}_r$ we excluded the case $c=0$, because of (iii) of Proposition \ref{p:Vr} we nevertheless have $V_r\subset \U_r$. Moreover, $\mathcal{V}_r$ and $\U_r$ are easily seen to be invariant w.r.t. the maps \eqref{e:rotation}.

\begin{proposition}\label{p:Ur}
For any $r>0$ we have
\begin{enumerate}
\item[(i)] $\U_r\subset\C\times\C$ is open;
\item[(ii)] $\K_{r'}\subset\overline{\U_r}$ for all $0\leq r'\leq r$;
\item[(iii)] For every $w\in \U_r$ and every $\eps>0$ there exists $r-\eps<r'<r$ and a laminate of finite order $\nu\in\mathcal{L}(\U_r)$ such that
$\bar{\nu}=w$ and $\supp\nu\subset \K_{r'}$.
\end{enumerate}
\end{proposition}

\begin{proof}

{\bf (i). }We note that the map $(z,c,\theta)\mapsto (ze^{i\theta},ce^{2i\theta})$ is a local immersion in the set $\{(z,c,\theta):\,|c|\neq 0\}$. Since $V_r$ is (relatively) open in $L$, it follows that $\mathcal{V}_r$ is open in $\C\times\C$. Openness of $\U_r$ then follows from Lemma \ref{l:openlchull}.

\smallskip

{\bf (ii). }By the invariance w.r.t. \eqref{e:rotation} it suffices to show that $\overline{\U_r}\cap L\supset \K_{r'}\cap L$. But $\overline{\U_r}\cap L\supset \overline{V_r}$. So the claim follows from Proposition \ref{p:Vr} (ii) and (iii).

{\bf (iii). }Since the set of laminates of finite order $\mathcal{L}(\U_r)$ is closed under splitting in $\U_r=\mathcal{V}_r^{lc}$, by using Lemma \ref{l:openlchull} and the invariance w.r.t. \eqref{e:rotation} we may reduce without loss of generality to the case $w\in V_r$. Choose $r-\eps<r'<r$ such that $w\in \overline{V_{r'}}$. By Proposition \ref{p:Vr} (iv)  there exists a laminate $\nu\in \mathcal{L}(\overline{V_{r'}})$ of finite order such that $\bar{\nu}=w$ and $\supp \nu\subset \K_{r'}\cap L$. Since $r'<r$, Proposition \ref{p:Vr} (ii) and (iii) imply $V_{r'}\subset \U_r$ and hence $\nu\in \mathcal{L}(\U_r)$. The statement of the Proposition follows.

\end{proof}

\bigskip

\begin{corollary}
The set $\U_r$ defined in \eqref{e:Ur} satisfies the perturbation property (P) and also property (*). 
\end{corollary}
 
\begin{proof}
Let $w_0\in \U_r$. Using Proposition \ref{p:Ur} (iii) for any $\eps>0$ there exists $r-\eps<r'<r$ and a laminate of finite order $\nu\in \mathcal{L}(\U_r)$ with barycenter $\bar{\nu}=w_0$ such that $\supp\nu\subset \K_{r'}$. Consequently, writing $w=(z,\zeta)$,
\begin{align*}
\int|w-w_0|^2\,d\nu(w)&\geq \int|z-z_0|^2\,d\nu(z,\zeta)\\&
=\int|z|^2-2\Re(z\bar{z_0})+|z_0|^2\,d\nu(z,\zeta)=r'-|z_0|^2
\end{align*}
Since $\eps>0$ is arbitrary and we have \eqref{e:Psi}, Proposition \ref{p:criterion} applies and implies property (P).

Property (*) is a direct consequence of Proposition \ref{p:Vr} (iii) and Proposition \ref{p:Ur} (ii). 

\end{proof}
 
%%%%%%%%%%%%%%%%%%%%%%%%%%%%%%%%%%%%%%
%%%%%%%%%%%%%%%%%%%%%%%%%%%%%%%%%%%%%%
 
\section{Failure of Property (P)}
\label{s:failureofP}

In this section we show that in the case $d=2$ the Property (P) fails for the interior of the convex hull of $\K_r$. In the language of compensated compactness this amounts to an additional non-trivial constraint on the relaxation - in the framework of gradient differential inclusions of the type $Du\in K$ \cite{KIRCHHEIM:habilitation,MUELLER-SVERAK,SZEKELYHIDI:SISSA-lectures} this amounts to the statement that the quasiconvex hull of $K$ is strictly smaller than the convex hull. We do not know what the (analogue of) the quasiconvex hull of $\K_r$ is in this case.

\smallskip

\begin{theorem}
Let $d=2$ and $\U_r:=\textrm{int }\K_r^{co}$. Then Property (P) is not valid.
\end{theorem}

\begin{proof}

{\bf 1. } We will treat the case $r=1$, the general case follows easily by scaling. To start with we will analyse the boundary $\partial \K_1^{co}$. Recalling the expression for $\K_1^{co}$ from \eqref{e:convexhull} we see that if $(\bar{v},\bar{u})\in \partial\K_1^{co}\setminus\K_1$, then, after using the maps \eqref{e:rotation} in the form $\theta\mapsto R_\theta(\bar{v}\otimes \bar{v}-\bar{u})R_\theta^T$ we have
\begin{equation}\label{e:reduction}
\bar{v}\otimes \bar{v}-\bar{u}=\begin{pmatrix}1/2 & 0\\ 0&\lambda\end{pmatrix}
\end{equation}
for some $\lambda<1/2$. Let $(\tilde v,\tilde u)\in \R^2\times\S_0^2$ be a direction (e.g. normalized so that $|\tilde v|=1$) such that
$(\bar{v}+t\tilde v,\bar{u}+t\tilde u)\in  \partial\K_1^{co}$ for all $|t|<\delta$ for some $\delta>0$. This amounts to
\begin{equation}\label{e:boundary}
(\bar{v}+t\tilde v)\otimes (\bar{v}+t\tilde v)-(\bar{u}+t\tilde u)=\begin{pmatrix}1/2&0\\0 &\lambda\end{pmatrix}+tA+t^2B\leq \begin{pmatrix}1/2&0\\ 0&1/2\end{pmatrix},
\end{equation}
where 
$$
A=\tilde v\otimes\bar{v}+\bar{v}\otimes \tilde v-\tilde u,\quad B=\tilde v\otimes \tilde v.
$$
In particular we require $\textrm{diag}(0,1/2-\lambda)-tA$ to be positive semidefinite for all sufficiently small $|t|$. Expanding $t\mapsto \det(\textrm{diag}(0,1/2-\lambda)-tA)$ in a quadratic polynomial, we obtain the necessary conditions $A_{11}=A_{12}=A_{21}=0$. Then \eqref{e:boundary} reduces to
$$
0\leq \begin{pmatrix}-t^2 \tilde v_1^2& -t^2 \tilde v_1\tilde v_2\\-t^2\tilde v_1\tilde v_2& \tfrac{1}{2}-\lambda-t A_{22}-t^2\tilde v_2^2\end{pmatrix}
$$
We deduce $\tilde v_1=0$. Plugging into the definition of $A$ and using that $A=\textrm{diag}(0,A_{22})$ we finally obtain
\begin{equation}\label{e:direction}
\tilde v=\begin{pmatrix}0\\1\end{pmatrix},\quad \tilde u=\begin{pmatrix}0&\bar{v}_1\\ \bar{v}_1&0\end{pmatrix}.
\end{equation}
Therefore the boundary of $\K_1^{co}$ at $(\bar{v},\bar{u})$ consists of a single line segment in the direction $(\tilde v,\tilde u)$. Observe that $(\tilde v,\tilde u)\notin\Lambda$, unless $\overline v_1=0$.

\medskip

%%%%%%%%%%%%%%%
{\bf 2. }We now argue by contradiction. Assume that the perturbation property (P) holds and let $\bar{w}=(\bar{v},\bar{u})\in \partial\K_1^{co}\setminus\K_1$, without loss of generality satisfying \eqref{e:reduction}. Assume further that $\bar{v}_1\neq 0$ (it is easy to see that such $\bar{w}\in \partial\K_1^{co}\setminus\K_1$ exists). 

Let $\bar{w}^{(k)}=(\bar{v}^{(k)},\bar{u}^{(k)})\in \textrm{int}\,\K_1^{co}$ be a sequence such that $\bar{w}^{(k)}\to \bar{w}$. Then there exists $\delta>0$ and for each $k\in\N$ there exists a subsolution $w^{(k)}\in C_c^{\infty}(Q;\R^2\times \S^2_0)$ such that $\bar{w}^{(k)}+w^{(k)}(x)\in \K_1^{co}$ and $\int_Q|w^{(k)}(x)|^2\,dx\geq \delta$. Define the probability measures $\nu_k$ on $\R^2\times \S^2_0\times\R$ by duality using the formula
\begin{equation}\label{e:youngmeasure}
\int f(v,u,q)\,d\nu_k(v,u,q):=\int_Qf(v^{(k)}(x),u^{(k)}(x),q^{(k)}(x))\,dx\qquad\forall\,f\in C_c(\R^2\times \S^2_0\times\R),
\end{equation}
where $q^{(k)}(x)$ is the associated pressure, i.e. the solution of the equation 
\begin{align*}
\Delta q^{(k)}&=-\textrm{div }\textrm{div }u^{(k)}\textrm{ on }Q\\
q^{(k)}&=0\textrm{ on }\partial Q.
\end{align*} 
(Here one should recall that the pressure in property (P) is required to satisfy Dirichlet boundary conditions).
Note that $\supp\nu_k\subset\K_1^{co}\times\R$ for all $k\in\N$.

Using the weak* sequential compactness of the dual space $C_c(\R^2\times \S^2_0\times\R)^*$ we obtain a weakly* convergent subsequence $\nu_k\overset{*}{\rightharpoonup}\nu$. We note in passing that the probability measure $\nu$ is a stationary measure-valued (sub)solution of the Euler equations. (c.f. \cite{DiPerna}). 

Using that $(\bar{v}^{(k)},\bar{u}^{(k)})+(v^{(k)},u^{(k)})\in \K_1^{co}$, we see that the sequence $(v^{(k)},u^{(k)})$ is uniformly bounded in $L^{\infty}(Q)$. Then, from the standard $L^p$-estimate $\|q^{(k)}\|_{L^p(Q)}\leq C_p\|u^{(k)}\|_{L^p(Q)}$ for any $p<\infty$, we deduce that the sequence $q^{(k)}$ is uniformly bounded in $L^p(Q)$ for any $p<\infty$. Consequently in \eqref{e:youngmeasure} one may extend to test functions $f\in C(\R^2\times \S^2_0\times\R)$ with at most polynomial growth. On the probability measure $\nu$ we deduce
\begin{equation}\label{e:info-on-nu}
\bar{\nu}=(\bar{v},\bar{u},\bar{q}),\quad \supp\nu\subset\K_1^{co}\times\R,\quad \int|(v,u)|^2\,d\nu(v,u,q)\geq \delta
\end{equation}
for some $\bar{q}\in\R$.
Here, $\bar \nu$ denotes the mean (barycenter) of the probability measure $\nu$.
 Since $(\bar{v},\bar{u})\in \partial\K_1^{co}$, we obtain, using {\bf 1.}, that $\supp\nu \subset\mathscr{L}\times\R$, where $\mathscr{L}\subset\R^2\times\S^2_0$ is the line through $(\bar{v},\bar{u})$ in the direction given by \eqref{e:direction}.

\medskip

{\bf 3. }Observe that the sequence $(v^{(k)},u^{(k)},q^{(k)})$ satisfies 
$\partial_1v^{(k)}_1+\partial_2v^{(k)}_2=0$ and
$$
\partial_1(u^{(k)}_{11}+q^{(k)})+\partial_2u^{(k)}_{12}=0,\quad \partial_1u^{(k)}_{12}+\partial_2(q^{(k)}-u^{(k)}_{11})=0.
$$
Using the div-curl lemma and standard tools from Young measure theory we deduce that $\nu$ commutes with the functions 
$$
g_1(v,u,q):=v_1(q-u_{11})-v_2u_{12},\quad g_2(v,u,q):=(u_{11}+q)(q-u_{11})-u_{12}^2,
$$
i.e. $\int g_i\,d\nu=g_i(\bar{\nu})$ for $i=1,2$ (c.f. \cite{DiPerna}, where this is referred to as the \emph{commutativity relation}). Hence $\nu$ also commutes with 
$$
g(v,u,q):=g_2(v,u,q)-2\bar{v}_1 g_1(v,u,q),
$$
$g$ being a linear combination of $g_1$ and $g_2$. 
However, on the support of $\nu$, i.e. on $\mathscr{L}\times\R$ the function $g$ becomes
$$
(t,s)\mapsto g(\bar{v}+t\tilde v,\bar{u}+t\tilde u,\bar{q}+s)=c_0+c_1s+c_2t+s^2+\bar{v}_1^2t^2,
$$
where $c_0=\bar{q}+\bar{v}_1^2(1/2+\lambda-2\bar{q})-1/4$, $c_1=2(\bar{q}-\bar{v}_1^2)$ and $c_2=2\bar{v}_1^2\bar{v}_2$. Here we have used the identities \eqref{e:reduction} and \eqref{e:direction}. Observe that, since we have assumed $\bar{v}_1\neq 0$, the function $(t,s)\mapsto g$ is strictly convex. Consequently,  
from Jensen's inequality we deduce that $\nu=\delta_{(\bar{v},\bar{u},\bar{q})}$ is a Dirac measure. 
This contradicts \eqref{e:info-on-nu}, thus concluding the proof.

\end{proof}

\section*{Acknowledgements}
The authors would like to thank Vladim\'\i r \v Sver\'ak for initial discussions and suggesting to consider stationary solutions, and Denis Serre for pointing out some mistakes in previous versions of this manuscript. Both authors acknowledge the support of ERC Grant Agreement No.~277993.

%%%%%%%%%%%%%%%%%%%%%%%%%%%%%%%%%%%%
%%%%%%%%%%%%%%%%%%%%%%%%%%%%%%%%%%%%
%%%%%%%%%%%%%%%%%%%%%%%%%%%%%%%%%%%%
%%%%%%%%%%%%%%%%%%%%%%%%%%%%%%%%%%%%

\end{document}